\documentclass[11pt]{amsart}
\usepackage{amsmath,amssymb,amsthm}
\usepackage[margin=1in]{geometry}

\usepackage{xcolor}
\usepackage[colorlinks=true,linkcolor=blue!55!black,citecolor=blue!55!black,urlcolor=blue!55!black]{hyperref}

\newtheorem{theorem}{Theorem}
\newtheorem{proposition}[theorem]{Proposition}
\newtheorem{lemma}[theorem]{Lemma}
\newtheorem{corollary}[theorem]{Corollary}
\newtheorem{definition}[theorem]{Definition}
\theoremstyle{remark}
\newtheorem{remark}[theorem]{Remark}

\title{The $a$-number of the Serre-type maximal curves of\\ Abd\'on--Bezerra--Quoos}
\author{}
\thanks{Email address: mshafighi@missouri.edu}
\date{}

\begin{document}
\maketitle
\vspace{0.5em}
\centerline{\itshape Mahyar Shafighi}
\vspace{6pt}
\centerline{\itshape Department of Mathematics, University of Missouri, Columbia, MO, USA}
\vspace{1em}

\begin{abstract}
For every prime power $q$, we determine the $a$-number of the Serre-type
maximal curves $X(q,n)$: $y^{q^2}-y=x^{(q^n+1)/(q+1)}$ over
$\mathbb{F}_{q^{2n}}$, $n\ge 3$ odd, introduced by Abd\'on, Bezerra and
Quoos as a generalization of Serre's curve $y^4+y=x^3$ over
$\mathbb{F}_{64}$. The
Weierstrass semigroup at the unique point at infinity gives an explicit
basis of holomorphic differentials, on which the Cartier operator acts by
monomials; this reduces the $a$-number to a lattice-point count, which we
evaluate in closed form.
\end{abstract}

\section{Introduction}\label{sec:intro}

The $a$-number is an invariant of the $p$-torsion group scheme of the
Jacobian of a curve in characteristic $p$ that complements the $p$-rank and measures how
degenerate that group scheme is: it is the largest integer $a$ such that $\alpha_p^{\,a}$ is a subgroup scheme of it,
equivalently the dimension of the kernel of the Cartier operator on the
space of holomorphic differentials. While the $p$-rank of
a curve can often be computed via a Deuring--Shafarevich formula for
$p$-group covers (see Crew~\cite[Corollary~1.8]{Crew}, and Subrao~\cite{Subrao}
for Artin--Schreier covers), no analogous general formula is known for the $a$-number, which has
been computed only in particular cases, for example for Fermat and Hurwitz
curves~\cite{MS}, Suzuki curves~\cite{FGMPW}, and hyperelliptic curves in
characteristic $2$~\cite{EP}. Each new family for which the $a$-number can be computed exactly is
therefore of independent interest. Throughout, $k$ denotes an algebraic
closure of $\mathbb{F}_p$, and the $a$-number of a curve defined over a
finite field is that of its base change to $k$.

The Serre-type maximal curves $X(q,n)$ studied in this paper, given by
\[
y^{q^2}-y=x^{(q^n+1)/(q+1)}
\]
over $\mathbb{F}_{q^{2n}}$ for $n\ge 3$ odd, were introduced by Abd\'on,
Bezerra and Quoos~\cite{ABQ} as a generalization of Serre's original curve
$y^4+y=x^3$ over $\mathbb{F}_{64}$ (the case $q=2,n=3$). They showed that
$X(q,n)$ is $\mathbb{F}_{q^{2n}}$-maximal, attaining the Hasse--Weil upper
bound on its number of rational points, and determined its genus
$g=(q-1)(q^n-q)/2$. Their work concerns the curve's point count and genus; its
$a$-number, a basic invariant of the $p$-torsion of its Jacobian, had not,
to our knowledge, previously been determined.

Studying the $a$-number of $X(q,n)$ is worthwhile for a concrete reason.
Since $X(q,n)$ is maximal, a theorem of Garcia and
Tafazolian shows that its Cartier operator is
nilpotent (indeed $C^{tn}=0$ for $q=p^t$; see
\cite[Theorem~3.3]{GarciaTafazolian}); in particular its $p$-rank,
which equals the rank of a
sufficiently high power of the Cartier operator, is zero for every member
of the family. The $p$-rank therefore carries no information
here, and the $a$-number becomes the natural finer invariant
left to distinguish one maximal curve from another. Our main result
determines the $a$-number of $X(q,n)$ for every prime power $q$ and every odd
$n\ge 3$.

\begin{theorem}\label{thm:main}
Let $q=p^t$ be a power of a prime $p$, let $n\ge 3$ be odd, and let
$X=X(q,n)$. Then
\[
a(X)=\frac{p-1}{12p(q+1)}\Big[(3q^2+p-2)\,q^n+q\big((p-2)q^2-(p+1)q-3p\big)\Big].
\]
\end{theorem}

The proof has three ingredients. The Weierstrass semigroup $\langle
q^2,m\rangle$ at the unique point at infinity, where $m=(q^n+1)/(q+1)$,
gives an explicit basis of holomorphic differentials
(Section~\ref{sec:curves}). On this basis the Cartier operator acts by
monomials: writing $y=(y^{q^2/p})^p-x^m$ and applying Lucas' theorem, the
image of each basis element is either $0$ or a single basis monomial, and
distinct basis elements have distinct images (Section~\ref{sec:cartier}).
The $a$-number is therefore the number of basis monomials $x^iy^j\,dx$
with $(i\bmod p)+(j\bmod p)\le p-2$, and Section~\ref{sec:count} evaluates
this lattice-point count in closed form, proving Theorem~\ref{thm:main}.
Section~\ref{sec:consequences} records special cases and checks.

\section{Preliminaries}\label{sec:prelim}

Let $X$ be a smooth projective curve of genus $g$ over $k$, and write
$J=\operatorname{Jac}(X)$ for its Jacobian. The $a$-number of $X$ is
\[
a(X) = \dim_k \operatorname{Hom}(\alpha_p, J[p]),
\]
where $J[p]$ is the kernel of multiplication by $p$ on $J$, and $\alpha_p$
is the kernel of the Frobenius endomorphism of the additive group scheme
$\mathbb{G}_a=\operatorname{Spec}k[T]$, that is,
$\alpha_p=\operatorname{Spec}(k[T]/(T^p))$.

The $a$-number can also be computed from the Dieudonn\'e module of $J[p]$.
By classical Dieudonn\'e theory (see Demazure~\cite{Demazure}), the functor
$G\mapsto D(G)$ is an anti-equivalence from finite commutative group
schemes over $k$ killed by $p$ to finite-dimensional $k$-vector spaces
equipped with a Frobenius-semilinear map $F$ and a Verschiebung-semilinear
map $V$ satisfying $FV=VF=0$, and $D(\alpha_p)=k$ with $F=V=0$. Put
$M=D(J[p])$; it has dimension $2g$, and $F$ and $V$ on $M$ are induced by
the relative Frobenius $F:J\to J^{(p)}$ and the Verschiebung
$V:J^{(p)}\to J$, whose composite $V\circ F$ is multiplication by $p$.
Since $D$ is contravariant, a homomorphism $\alpha_p\to J[p]$ corresponds to
a linear map $M\to k$ vanishing on $FM+VM$, so
\[
a(X)=\dim_k M/(FM+VM).
\]
Moreover, $J[p]$ is the $p$-torsion of the $p$-divisible group of $J$, so
it is a truncated Barsotti--Tate group of level $1$, which means (see
Demazure~\cite{Demazure} or Moonen~\cite[Section~2.1]{Moonen}) that $FM=\operatorname{Ker}(V)$ and
$VM=\operatorname{Ker}(F)$. By rank--nullity,
$\dim\operatorname{Ker}(F)+\dim FM=2g$, hence
$\dim\operatorname{Ker}(F)+\dim\operatorname{Ker}(V)=2g$, and therefore
\[
\dim M/(FM+VM)=2g-\dim\big(\operatorname{Ker}(F)+\operatorname{Ker}(V)\big)
=\dim\big(\operatorname{Ker}(F)\cap\operatorname{Ker}(V)\big).
\]
Thus
\[
a(X) = \dim_k\big(\operatorname{Ker}(F)\cap\operatorname{Ker}(V)\big),
\]
where $F$ and $V$ act on $M$.

The $a$-number can be computed from the Cartier operator; we first recall
its definition.

The Cartier operator on differential forms in characteristic $p$ was
introduced by Cartier~\cite{Cartier}.

\begin{definition}\label{def:cartier}
Let $K$ be a field of characteristic $p>0$ with $[K:K^p]=p$, and let
$x\in K\setminus K^p$. Then $1,x,\dots,x^{p-1}$ is a basis of $K$ over
$K^p$, and $\Omega_{K/K^p}$ is the one-dimensional $K$-vector space spanned
by $dx$. Hence every $\omega\in\Omega_{K/K^p}$ can be written uniquely as
\[
\omega = (f_0^p+f_1^px+\dots+f_{p-1}^px^{p-1})\,dx,\qquad f_0,\dots,f_{p-1}\in K.
\]
The Cartier operator $C$ sends $\omega$ to
\[
C(\omega)=f_{p-1}\,dx.
\]
\end{definition}

\emph{Note: The definition does not depend on the choice of $x$.}

We use this for the function field $K=k(X)$ of the curve $X$, which
satisfies $[K:K^p]=p$ because $k$ is perfect. Since $k=k^p\subset K^p$, the
space $\Omega_{K/K^p}$ equals $\Omega_{K/k}$, the space of differentials of
$X$. Every
differential is then $h\,dx$, for any chosen $x\in K\setminus K^p$, with $h=\sum_{r=0}^{p-1}f_r^px^r$ for unique
$f_r\in K$, so that $C(h\,dx)=f_{p-1}\,dx$. The operator is
$p^{-1}$-linear, that is, $C(\lambda^p\omega)=\lambda C(\omega)$ for
$\lambda\in k$.

By a theorem of Oda~\cite[Theorem~5.10 and Corollaries~5.11, 5.12]{Oda}, there is a canonical isomorphism
identifying the Dieudonn\'e module $D(J[p])$ with $H^1_{\mathrm{dR}}(X)$,
compatible with the Hodge filtration
\[
0\to H^0(X,\Omega^1)\to H^1_{\mathrm{dR}}(X)\to H^1(X,\mathcal{O}_X)\to 0.
\]
Under this identification $F$ vanishes identically on $H^0(X,\Omega^1)$ and
induces a Frobenius-semilinear map $H^1(X,\mathcal{O}_X)\to
H^1_{\mathrm{dR}}(X)$, while $V$ has image exactly $H^0(X,\Omega^1)$ and restricts there to the
Cartier operator $C$.

We can now relate this to the Cartier operator. Since $J[p]$ is a
truncated Barsotti--Tate group of level $1$, we have
$\operatorname{Im}(V)=\operatorname{Ker}(F)$ on $M$. Combined with $\operatorname{Im}(V)=H^0(X,\Omega^1)$ above, this forces
\[
\operatorname{Ker}(F) = H^0(X,\Omega^1).
\]
Hence
\[
\operatorname{Ker}(F)\cap\operatorname{Ker}(V)
= H^0(X,\Omega^1)\cap\operatorname{Ker}(V)
= \operatorname{Ker}\big(V|_{H^0(X,\Omega^1)}\big) = \operatorname{Ker}(C),
\]
where the last equality holds because the restriction of $V$ to
$H^0(X,\Omega^1)=\operatorname{Ker}(F)$ is the Cartier operator $C$. Therefore the $a$-number is given by the Cartier operator:
\[
a(X) = \dim_k\operatorname{Ker}(C) = g-\operatorname{rank}(C),
\]
which is the form we compute throughout the paper. In particular $C$ maps
$H^0(X,\Omega^1)$ into itself.

\section{The curves $X(q,n)$ and their holomorphic differentials}\label{sec:curves}

Let $p$ be a prime, $q=p^t$ a power of $p$, $n\ge 3$ an odd integer, and
put
\[
Q:=q^2,\qquad m:=\frac{q^n+1}{q+1}=\sum_{k=0}^{n-1}(-q)^k .
\]
Following Abd\'on--Bezerra--Quoos~\cite{ABQ}, let $X=X(q,n)$ be the
smooth projective curve over $k$ with affine plane model
\[
F(x,y)=y^{Q}-y-x^m=0,
\]
defined over $\mathbb{F}_{q^{2n}}$. When $q=2,n=3$ this is Serre's curve
$y^4+y=x^3$ over $\mathbb{F}_{64}$. The curve $X$ is
$\mathbb{F}_{q^{2n}}$-maximal of genus
\[
g=\frac{(q-1)(q^n-q)}{2}.
\]

\begin{lemma}\label{lem:m}
\begin{enumerate}
\item[(a)] $m\equiv 1\pmod q$; in particular $m\equiv 1\pmod p$ and
$\gcd(m,Q)=1$.
\item[(b)] $m=Q\mu+1-q$, where $\mu=(q^{n-2}+1)/(q+1)=\sum_{k=0}^{n-3}(-q)^k$
is an integer with $\mu\equiv 1\pmod p$.
\item[(c)] $g=(Q-1)(m-1)/2$, hence $2g-2=Qm-Q-m-1$.
\end{enumerate}
\end{lemma}

\begin{proof}
(a) Every term of $\sum_{k=0}^{n-1}(-q)^k$ with $k\ge 1$ is divisible by $q$.
(b) We have $m-(1-q)=\sum_{k=2}^{n-1}(-q)^k=q^2\sum_{k=0}^{n-3}(-q)^k$, and
since $n-2$ is odd, $\sum_{k=0}^{n-3}(-q)^k=(1-(-q)^{n-2})/(1+q)=(q^{n-2}+1)/(q+1)$.
As in (a), this sum is $\equiv 1\pmod q$. (c) $m-1=(q^n-q)/(q+1)$, so
$(Q-1)(m-1)/2=(q-1)(q^n-q)/2=g$; then
$2g-2=(Q-1)(m-1)-2=Qm-Q-m-1$.
\end{proof}

\begin{proposition}\label{prop:infty}
\begin{enumerate}
\item[(a)] The affine curve $F(x,y)=0$ is smooth, and $F_y=-1$.
\item[(b)] The function $x$ has a pole at exactly one point $P_\infty$ of $X$,
and $v_{P_\infty}(x)=-Q$, $v_{P_\infty}(y)=-m$.
\item[(c)] $X\setminus\{P_\infty\}$ is the affine curve $F=0$.
\end{enumerate}
\end{proposition}

\begin{proof}
(a) Since $Q=q^2$ is divisible by $p=\operatorname{char}(k)$,
$F_y=Qy^{Q-1}-1=-1$ identically, so the gradient of $F$ never vanishes.

(b) Let $f$ be a monic irreducible factor of $F$ in $k(x)[y]$, of degree
$d$, and let $L=k(x)[y]/(f)$, in which we again denote the class of $y$ by
$y$. Let $P$ be a place of $L$ above $x=\infty$ and $e:=-v_P(x)>0$ its
ramification index; then $e\le[L:k(x)]=d$. Since $v_P(x^m)=-em<0$, the
equation $y^Q-y=x^m$ forces $v_P(y)<0$ and $Qv_P(y)=-em$. Hence $Q\mid em$,
so $Q\mid e$ because $\gcd(m,Q)=1$ (Lemma~\ref{lem:m}(a)). Thus $d\ge e\ge Q$,
so $f=F$: the polynomial $F$ is irreducible over $k(x)$, hence (being monic in
$y$) over $k$, and $[k(X):k(x)]=Q$. Now apply the same argument to
$L=k(X)$: every place $P$ above $\infty$ has $e=Q=[k(X):k(x)]$, so there is
only one such place, and $v_P(y)=-m$.

(c) By (b), every place of $k(X)$ other than $P_\infty$ has $v(x)\ge 0$, hence
$v(y)\ge 0$ because $y$ is integral over $k[x]$. Such places correspond to
the points of the affine curve $F=0$, which is smooth by (a).
\end{proof}

\begin{proposition}\label{prop:basis}
Let $B=\{(i,j)\in\mathbb{Z}^2 : i\ge 0,\ 0\le j\le Q-1,\ Qi+jm\le 2g-2\}$.
Then $\operatorname{div}(dx)=(2g-2)P_\infty$, and
\[
\{\,x^iy^j\,dx : (i,j)\in B\,\}
\]
is a basis of $H^0(X,\Omega^1)$. In particular $|B|=g$.
\end{proposition}

\begin{proof}
Since $F_y\ne 0$ at every affine point, $x-x(P)$ is a uniformizer at each
affine point $P$, so $dx$ has neither zero nor pole there. Hence
$\operatorname{div}(dx)$ is supported at $P_\infty$, and has degree
$2g-2$, which gives the first claim.

Let $R=k[x,y]/(F)$ be the coordinate ring of the affine curve. Since the
relation $F=0$ is monic in $y$, $R$ is a free $k[x]$-module with basis
$1,y,\dots,y^{Q-1}$. By Proposition~\ref{prop:infty}(a), $R$ is regular,
hence integrally closed, so $R$ is the ring of functions on $X$ regular
away from $P_\infty$. Thus $h\,dx$ is holomorphic if and only if $h\in R$
and $v_{P_\infty}(h)\ge-(2g-2)$.

By Proposition~\ref{prop:infty}(b), $v_{P_\infty}(x^iy^j)=-(Qi+jm)$. These
values are pairwise distinct for $i\ge 0$, $0\le j\le Q-1$: if
$Qi+jm=Qi'+j'm$, then $Q\mid(j-j')m$, so $Q\mid j-j'$, so $j=j'$ and
$i=i'$. Since a valuation of a sum of terms with pairwise distinct
valuations is the minimum of their valuations, $h=\sum c_{ij}x^iy^j$
satisfies $v_{P_\infty}(h)\ge-(2g-2)$ if and only if every monomial
occurring in it does. So the stated monomials span $H^0(X,\Omega^1)$, and
they are linearly independent, so $|B|=\dim H^0(X,\Omega^1)=g$. Moreover, the
pole orders $Qi+jm$ at $P_\infty$ of the functions regular away from
$P_\infty$ are exactly the elements of the numerical semigroup
$\langle Q,m\rangle$ (write $aQ+bm$ with $b=cQ+j$, $0\le j<Q$, as
$(a+cm)Q+jm$); this is the Weierstrass semigroup at $P_\infty$.
\end{proof}

\section{The Cartier operator on the basis}\label{sec:cartier}

For integers $i\ge 0$ and $0\le j\le Q-1$ write $i_0,j_0\in\{0,\dots,p-1\}$
for the residues of $i,j$ modulo $p$, and $j=pj_1+j_0$.

\begin{lemma}\label{lem:cartier}
\begin{enumerate}
\item[(a)] If $i_0+j_0\le p-2$, then $C(x^iy^j\,dx)=0$.
\item[(b)] If $i_0+j_0\ge p-1$, put $\nu_0=i_0+j_0-(p-1)$ and
$N=i+m(j-\nu_0)$. Then $N\equiv p-1\pmod p$ and
\[
C(x^iy^j\,dx)=\epsilon\,x^{S}y^{E}\,dx,\qquad
S=\frac{N-(p-1)}{p}-mj_1,\quad E=\frac{\nu_0Q}{p}+j_1,
\]
where $\epsilon=(-1)^{j-\nu_0+j_1}\binom{j_0}{\nu_0}$ is a nonzero element
of $\mathbb{F}_p$, $S\ge 0$ and $0\le E\le Q-1$.
\end{enumerate}
\end{lemma}

\begin{proof}
Put $w=y^{Q/p}$. Since $y^Q=y+x^m$, we have $y=w^p-x^m$, hence
\[
x^iy^j=\sum_{\nu=0}^{j}(-1)^{j-\nu}\binom{j}{\nu}\,w^{p\nu}\,x^{N_\nu},
\qquad N_\nu:=i+m(j-\nu).
\]
Write $N_\nu=ps_\nu+r_\nu$ with $0\le r_\nu<p$. The coefficients
$c_\nu=(-1)^{j-\nu}\binom j\nu$ lie in $\mathbb{F}_p$, so $c_\nu^p=c_\nu$ and
\[
x^iy^j=\sum_{r=0}^{p-1}\Big(\sum_{\nu:\,r_\nu=r}c_\nu w^{\nu}x^{s_\nu}\Big)^p x^r .
\]
By Definition~\ref{def:cartier} (applicable because $x\notin K^p$, as
$dx\ne 0$ by Proposition~\ref{prop:basis}),
\[
C(x^iy^j\,dx)=\sum_{\nu:\,r_\nu=p-1}c_\nu\,w^{\nu}x^{s_\nu}\,dx.
\]
We first determine which $\nu$ contribute, and then use Lucas' theorem to
control their coefficients. Since $m\equiv 1\pmod p$ (Lemma~\ref{lem:m}),
$N_\nu\equiv i+j-\nu\pmod p$,
so $r_\nu=p-1$ if and only if $\nu\equiv i_0+j_0+1\pmod p$. By Lucas'
theorem, for $\nu=\nu'+pb$ with $0\le\nu'<p$ we have
$\binom j\nu\equiv\binom{j_0}{\nu'}\binom{j_1}{b}\pmod p$.

(a) If $i_0+j_0\le p-2$, then every surviving $\nu$ has $\nu'=i_0+j_0+1>j_0$,
so $\binom{j_0}{\nu'}=0$ and the sum is empty: $C(x^iy^j\,dx)=0$.

(b) If $i_0+j_0\ge p-1$, the surviving $\nu$ are those with $\nu'=\nu_0$
(note $0\le\nu_0\le j_0$ because $i_0\le p-1$). The condition $\nu\le j$
becomes $0\le b\le j_1$. For such $\nu=\nu_0+pb$ we have
$N_\nu=N-pmb$, so $s_\nu=s_0-mb$ with $s_0=(N-(p-1))/p$, and
$N\equiv i+j-\nu_0\equiv p-1\pmod p$, so $s_0$ is an integer. Moreover
$w^\nu=w^{\nu_0}(w^p)^b=y^{\nu_0Q/p}(y+x^m)^b$, and
$(-1)^{pb}=(-1)^b$ in $\mathbb{F}_p$. Therefore
\begin{align*}
C(x^iy^j\,dx)&=(-1)^{j-\nu_0}\binom{j_0}{\nu_0}x^{s_0}y^{\nu_0Q/p}
\sum_{b=0}^{j_1}\binom{j_1}{b}(-1)^b\Big(\frac{y+x^m}{x^m}\Big)^b\,dx\\
&=(-1)^{j-\nu_0}\binom{j_0}{\nu_0}x^{s_0}y^{\nu_0Q/p}
\Big(1-\frac{y+x^m}{x^m}\Big)^{j_1}dx\\
&=(-1)^{j-\nu_0+j_1}\binom{j_0}{\nu_0}\,x^{s_0-mj_1}\,y^{\nu_0Q/p+j_1}\,dx .
\end{align*}
Here $S=s_0-mj_1=s_\nu$ for $\nu=\nu_0+pj_1\le j$, and $N_\nu\ge 0$ with
$N_\nu\equiv p-1\pmod p$ forces $N_\nu\ge p-1$, so $S\ge 0$. Finally
$0\le j_1\le Q/p-1$ (as $j\le Q-1$ and $p\mid Q$) and $\nu_0\le p-1$, so
$E\le (p-1)Q/p+Q/p-1=Q-1$. The scalar $\epsilon$ is nonzero because
$0\le\nu_0\le j_0<p$.
\end{proof}

Let $B_-=\{(i,j)\in B: i_0+j_0\le p-2\}$ and
$B_+=B\setminus B_-=\{(i,j)\in B: i_0+j_0\ge p-1\}$.

\begin{proposition}\label{prop:rank}
The Cartier operator kills $x^iy^j\,dx$ for $(i,j)\in B_-$ and sends the
elements of $B_+$ to nonzero multiples of pairwise distinct basis elements.
Consequently
\[
\operatorname{Ker}(C)=\operatorname{span}\{x^iy^j\,dx:(i,j)\in B_-\},
\qquad a(X)=|B_-|=\#\{(i,j)\in B:\ i_0+j_0\le p-2\}.
\]
\end{proposition}

\begin{proof}
The first assertion about $B_-$ is Lemma~\ref{lem:cartier}(a). For
$(i,j)\in B_+$, Lemma~\ref{lem:cartier}(b) gives $C(x^iy^j\,dx)=\epsilon
x^Sy^E\,dx$. Since $C$ maps $H^0(X,\Omega^1)$ into itself (Section~\ref{sec:prelim})
and $0\le E\le Q-1$, the criterion in the proof of Proposition~\ref{prop:basis}
(a monomial $x^Sy^E\,dx$ with $0\le E\le Q-1$ is holomorphic if and only if
$QS+Em\le 2g-2$) shows $(S,E)\in B$.

We show that $(i,j)\mapsto(S,E)$ is injective on $B_+$. Suppose two pairs
have the same $(S,E)$. Since $0\le j_1\le Q/p-1$ and $0\le\nu_0\le p-1$,
the equation $E=\nu_0Q/p+j_1$ determines $\nu_0$ and $j_1$. From the
definition of $S$,
\[
pS=N-(p-1)-pmj_1=i+m(j_0-\nu_0)-(p-1),
\]
so $i=pS+(p-1)-m(j_0-\nu_0)$ and $j=pj_1+j_0$; thus $(i,j)$ is determined by
$j_0$ once $(S,E)$ is fixed. Substituting,
\[
Qi+jm=Q(pS+p-1)+mQ\nu_0+mpj_1-m(Q-1)\,j_0,
\]
so two preimages with different $j_0$ would have values of $Qi+jm$
differing by a nonzero multiple of $m(Q-1)$. But $m(Q-1)>(Q-1)(m-1)-2=2g-2$
(Lemma~\ref{lem:m}(c)), while every $(i,j)\in B$ has
$0\le Qi+jm\le 2g-2$. Hence $j_0$, and with it $(i,j)$, is unique.

Now let $\omega=\sum_{(i,j)\in B}c_{ij}\,x^iy^j\,dx$. By $p^{-1}$-linearity (recall that $k$ is perfect, so each $c_{ij}$ has a
unique $p$-th root in $k$),
$C(\omega)=\sum_{(i,j)\in B_+}c_{ij}^{1/p}\,\epsilon_{ij}x^{S}y^{E}\,dx$, a
combination of distinct basis elements with coefficients
$c_{ij}^{1/p}\epsilon_{ij}$. So $C(\omega)=0$ if and only if $c_{ij}=0$ for
all $(i,j)\in B_+$, which gives the description of $\operatorname{Ker}(C)$.
The formula for $a(X)=\dim\operatorname{Ker}(C)$ follows from
Section~\ref{sec:prelim}.
\end{proof}

\section{Counting: proof of Theorem~\ref{thm:main}}\label{sec:count}

By Proposition~\ref{prop:rank} it remains to evaluate $|B_-|$.

\begin{proof}[Proof of Theorem~\ref{thm:main}]
\emph{Step 1: reduction to a sum over $\ell$.} For $(i,j)\in B$ put
$\ell=Q-1-j$, so $0\le\ell\le Q-1$. By Lemma~\ref{lem:m}(c), the condition
$Qi+jm\le 2g-2=Qm-Q-m-1$ is equivalent to $Qi\le \ell m-Q-1$. For $\ell=0$
there is no $i\ge 0$. For $1\le\ell\le Q-1$ we have $Q\nmid \ell m$ (as
$\gcd(m,Q)=1$), so the admissible $i$ are $0\le i<c_\ell$, where
$c_\ell=\lfloor \ell m/Q\rfloor$. Since $p\mid Q$, $j\equiv-1-\ell\pmod p$,
so $j_0=p-1-\ell_0$ with $\ell_0$ the residue of $\ell$ modulo $p$. Thus
$i_0+j_0\le p-2$ if and only if $i_0<\ell_0$, and
\[
a(X)=\sum_{\ell=1}^{Q-1}\#\{0\le i<c_\ell:\ i\bmod p<\ell_0\}
=\sum_{\ell=1}^{Q-1}\Big[\ell_0\Big\lfloor\frac{c_\ell}{p}\Big\rfloor+\min(\rho_\ell,\ell_0)\Big],
\]
where $\rho_\ell$ is the residue of $c_\ell$ modulo $p$ (we used that
$\#\{0\le i<c:\ i\bmod p<r\}=r\lfloor c/p\rfloor+\min(c\bmod p,\,r)$:
each of the $\lfloor c/p\rfloor$ full blocks of $p$ consecutive integers
contains exactly $r$ such $i$, and the last partial block, of length
$c\bmod p$, contains $\min(c\bmod p,\,r)$ of them).

\emph{Step 2: the dependence on $n$.} By Lemma~\ref{lem:m}(b),
$\ell m=\ell Q\mu+\ell(1-q)$, so
\[
c_\ell=\ell\mu-h_\ell,\qquad h_\ell:=\Big\lceil\frac{\ell(q-1)}{Q}\Big\rceil .
\]
We claim $h_\ell=\lceil \ell/(q+1)\rceil$ for $1\le\ell\le Q-1$. Write
$\ell=(q+1)(h-1)+s$ with $1\le s\le q+1$ and $1\le h\le q-1$; then
\[
\frac{\ell(q-1)}{Q}=(h-1)+\frac{s(q-1)-(h-1)}{Q},\qquad
0<s(q-1)-(h-1)\le q^2-1<Q,
\]
so $h_\ell=h=\lceil\ell/(q+1)\rceil$. Since $\mu\equiv 1\pmod p$, we get
$c_\ell\equiv\ell-h_\ell\pmod p$, i.e.\ $\rho_\ell$ is the residue of
$\ell-h_\ell$, and $\lfloor c_\ell/p\rfloor=(c_\ell-\rho_\ell)/p$. Hence
\[
a(X)=\frac{\mu}{p}\,S_1+S_2,\qquad
S_1=\sum_{\ell=1}^{Q-1}\ell_0\,\ell,\qquad
S_2=\sum_{\ell=1}^{Q-1}\Big[\min(\rho_\ell,\ell_0)-\frac{\ell_0(h_\ell+\rho_\ell)}{p}\Big],
\]
where $S_1,S_2$ depend only on $q$ and not on $n$.

\emph{Step 3: computing $S_1$.} Write $\ell=pu+r$ with $0\le u\le Q/p-1$ and
$0\le r\le p-1$ (the term $\ell=0$ contributes $0$), and put $Q'=Q/p$. Using
$\sum r=p(p-1)/2$, $\sum r^2=(p-1)p(2p-1)/6$ and $\sum u=Q'(Q'-1)/2$,
\begin{align*}
S_1&=\sum_{u,r}r(pu+r)=p\cdot\frac{p(p-1)}{2}\cdot\frac{Q'(Q'-1)}{2}
+Q'\,\frac{(p-1)p(2p-1)}{6}\\
&=\frac{(p-1)\,pQ'\,(3pQ'+p-2)}{12}=\frac{(p-1)q^2(3q^2+p-2)}{12},
\end{align*}
since $pQ'=q^2$.

\emph{Step 4: computing $S_2$.} Write $\ell=(q+1)(h-1)+s+1$ with
$1\le h\le q-1$ and $0\le s\le q$. Then $h_\ell=h$, and since $p\mid q$,
$\ell-h=q(h-1)+s\equiv s\pmod p$ and $\ell\equiv h+s\pmod p$. So with
$\rho=s\bmod p$ and $\kappa=(h+\rho)\bmod p$,
\[
S_2=\sum_{h=1}^{q-1}\sum_{s=0}^{q}T,\qquad T=\min(\rho,\kappa)-\frac{\kappa(h+\rho)}{p}.
\]
For fixed $h$, as $s$ runs over $0,\dots,q-1$ each residue $\rho$ occurs $q/p$
times, and $s=q$ gives $\rho=0$. Write $h=pu+\varepsilon$ with
$0\le\varepsilon\le p-1$. For $\rho<p-\varepsilon$ we have $\kappa=\varepsilon+\rho$,
and for $\rho\ge p-\varepsilon$ we have $\kappa=\varepsilon+\rho-p$; one finds
\[
\sum_{\rho=0}^{p-1}\min(\rho,\kappa)=\tfrac{(p-\varepsilon)(p-\varepsilon-1)}{2}+\tfrac{\varepsilon(\varepsilon-1)}{2},
\]
and, writing $h+\rho=pu+\kappa+p\delta$ with $\delta=1$ if $\rho\ge p-\varepsilon$
and $0$ otherwise,
\[
\sum_{\rho=0}^{p-1}\frac{\kappa(h+\rho)}{p}
=u\frac{p(p-1)}{2}+\frac{(p-1)(2p-1)}{6}+\frac{\varepsilon(\varepsilon-1)}{2}.
\]
Therefore
\[
S_2=\sum_{h=1}^{q-1}\Theta(h),\qquad
\Theta(h)=\frac qp\Big[\tfrac{(p-\varepsilon)(p-\varepsilon-1)}{2}-u\tfrac{p(p-1)}{2}
-\tfrac{(p-1)(2p-1)}{6}\Big]-\frac{\varepsilon h}{p},
\]
where the last term is $T$ at $\rho=0$, i.e.\ the contribution of $s=q$.
Now sum over $h=0,\dots,q-1$ (so $u$ runs over $0,\dots,q/p-1$ and
$\varepsilon$ over $0,\dots,p-1$) and subtract the term $h=0$. Writing
$N'=q/p$, the pieces are
\begin{align*}
\sum_{h=0}^{q-1}\tfrac{(p-\varepsilon)(p-\varepsilon-1)}{2}&=N'\tfrac{(p+1)p(p-1)}{6}, &
\sum_{h=0}^{q-1}u&=\tfrac{pN'(N'-1)}{2},\\
\sum_{h=0}^{q-1}\tfrac{(p-1)(2p-1)}{6}&=\tfrac{q(p-1)(2p-1)}{6}, &
\sum_{h=0}^{q-1}\varepsilon h&=\tfrac{p(p-1)}{2}\cdot\tfrac{pN'(N'-1)}{2}+N'\tfrac{(p-1)p(2p-1)}{6},
\end{align*}
and $\Theta(0)=\frac qp\cdot\frac{(p-1)(p+1)}{6}$. Collecting terms gives
\[
\sum_{h=0}^{q-1}\Theta(h)=\frac{p-1}{12p}\Big[-3q^3+(p+1)q^2-(p-2)q\Big],\qquad
\Theta(0)=\frac{p-1}{12p}\cdot 2q(p+1),
\]
hence
\[
S_2=\frac{(p-1)\,q\,\big(-3q^2+(p+1)q-3p\big)}{12p}.
\]

\emph{Step 5: conclusion.} By Lemma~\ref{lem:m}(b), $\mu=(q^{n-2}+1)/(q+1)$, so
\begin{align*}
a(X)&=\frac{q^{n-2}+1}{p(q+1)}\cdot\frac{(p-1)q^2(3q^2+p-2)}{12}
+\frac{(p-1)\,q\,\big(-3q^2+(p+1)q-3p\big)}{12p}\\
&=\frac{p-1}{12p(q+1)}\Big[(3q^2+p-2)(q^n+q^2)+q(q+1)\big(-3q^2+(p+1)q-3p\big)\Big].
\end{align*}
Expanding, $(3q^2+p-2)q^2+q(q+1)(-3q^2+(p+1)q-3p)=q\big((p-2)q^2-(p+1)q-3p\big)$,
which gives the formula in Theorem~\ref{thm:main}.
\end{proof}

\section{Consequences and checks}\label{sec:consequences}

\begin{corollary}\label{cor:q2}
For $q=2$ and every odd $n\ge 3$, $a(X(2,n))=(2^n-2)/6=g/3$, where
$g=2^{n-1}-1$.
\end{corollary}

\begin{corollary}\label{cor:prime}
For a prime $q=p$ and odd $n\ge 3$,
$a(X(p,n))=\dfrac{(p-1)\big[(3p-2)p^{n-1}+p(p-4)\big]}{12}$.
\end{corollary}

Table~\ref{tab:values} lists some values of $a(X(q,n))$.

\begin{table}[h]
\centering
\begin{tabular}{rrrrr}
$q$ & $n$ & $g$ & $a(X(q,n))$ & $a/g$\\
\hline
2 & 3 & 3 & 1 & 0.3333\\
2 & 5 & 15 & 5 & 0.3333\\
2 & 7 & 63 & 21 & 0.3333\\
2 & 9 & 255 & 85 & 0.3333\\
3 & 3 & 24 & 10 & 0.4167\\
3 & 5 & 240 & 94 & 0.3917\\
3 & 7 & 2184 & 850 & 0.3892\\
4 & 3 & 90 & 25 & 0.2778\\
4 & 5 & 1530 & 409 & 0.2673\\
5 & 3 & 240 & 110 & 0.4583\\
5 & 5 & 6240 & 2710 & 0.4343\\
7 & 3 & 1008 & 476 & 0.4722\\
8 & 3 & 1764 & 454 & 0.2574\\
9 & 3 & 2880 & 990 & 0.3438
\end{tabular}
\caption{Values of $a(X(q,n))$ from Theorem~\ref{thm:main}.}
\label{tab:values}
\end{table}

\begin{remark}
We also tested Theorem~\ref{thm:main} numerically. For each of the 39 pairs
$(q,n)$ with $n\ge 3$ odd, $q\le 32$ a prime power, and genus below $10^6$
(the largest genus being $805\,200$), we computed the Cartier operator on
the basis of Proposition~\ref{prop:basis} directly from
Definition~\ref{def:cartier}, expanding each differential in the $p$-basis
$1,x,\dots,x^{p-1}$ using only the relation $y=(y^{q^2/p})^p-x^m$ and the
binomial theorem, and not the monomial formula of Lemma~\ref{lem:cartier}.
In every case the rank over $\mathbb{F}_p$ of the resulting matrix gives the
value in Theorem~\ref{thm:main}. For 173 pairs with $q\le 127$ and
$q^n\le 10^{13}$ we also compared the count $|B_-|$ of
Proposition~\ref{prop:rank} with the formula of Theorem~\ref{thm:main}. The
code is available at \url{https://github.com/mahyarshafighi/abq-a-number-code}.
These computations are not used in the proofs. In our first attempts at this
problem, the cases $q=2$, $n=3,5,7,9$ and $q=3$, $n=3$ were computed
differently, by resolving the singularity at infinity of the plane model,
determining the adjoint curves in the sense of
Gorenstein~\cite{Gorenstein}, and applying the St\"ohr--Voloch formula for
the Cartier operator~\cite{StohrVoloch}. This approach gives the same values
as Theorem~\ref{thm:main}.
\end{remark}

\end{document}